\documentclass[12pt]{amsart}

\usepackage{amssymb}

\usepackage{anysize}
\marginsize{2cm}{2cm}{2cm}{2cm}

\renewcommand{\le}{\leqslant}
\renewcommand{\ge}{\geqslant}

\usepackage{versions}
\usepackage{hyperref}
\hypersetup{
  colorlinks   = true, %Colours links instead of ugly boxes
  urlcolor     = blue, %Colour for external hyperlinks
  linkcolor    = blue, %Colour of internal links
  citecolor   = red %Colour of citations
}
\usepackage{cancel}
\usepackage{yhmath}
\usepackage{enumitem}
\usepackage{stackengine}

\def\b{\stackinset{c}{+0.8pt}{t}{-2pt}{\textit{\u{}}}{$b$}}
\newtheorem{theorem}{Theorem}[section]
\newtheorem{proposition}[theorem]{Proposition}
\newtheorem{lemma}[theorem]{Lemma}

\newtheorem{corollary}[theorem]{Corollary}

\theoremstyle{definition}
\newtheorem{definition}[theorem]{Definition}

\theoremstyle{remark}
\newtheorem{remark}[theorem]{Remark}

\numberwithin{equation}{section}

\DeclareMathOperator{\vol}{vol}
\DeclareMathOperator{\sys}{sys}

\newcommand*{\where}{\ \ifnum\currentgrouptype=16 \middle\fi|\ }

\renewcommand{\epsilon}{\varepsilon}
\renewcommand{\phi}{\varphi}
\renewcommand{\kappa}{\varkappa}
\renewcommand{\theta}{\vartheta}

\def\Z{{\mathbb Z}}
\def\R{{\mathbb R}}

\title{Systolic inequalities for the number of vertices}

\author{Sergey Avvakumov{$^\spadesuit$}}
\author{Alexey Balitskiy{$^\clubsuit$}}
\author{Alfredo Hubard{$^\diamondsuit$}}
\author{Roman Karasev{$^\heartsuit$}}

\thanks{{$^\spadesuit$} Supported by the European Research Council under the European Union's Seventh Framework Programme ERC Grant agreement ERC StG 716424 -- CASe}
\thanks{{$^\clubsuit$} This material is based in part upon work supported by the National Science Foundation under Grant No. DMS-1926686}

\thanks{{$^\diamondsuit$} Supported by ANR-17-CE40-0033 (SoS), ANR-17-CE40-0018 (CAAPS), ANR-19-CE40-0014 (Min-max) and ERC-339025 (GUDHI)}

\address{Sergey Avvakumov, Department of Mathematical Sciences, University of Copenhagen, Universitetspark 5, 2100 Copenhagen, Denmark}
\email{savvakumov@gmail.com}

\address{Alexey Balitskiy, Institute for Advanced Study, 1 Einstein Dr, Princeton, NJ 08540, USA
\newline \indent
Institute for Information Transmission Problems RAS, Bolshoy Karetny per. 19, Moscow, Russia 127994}
\email{abalitskiy@ias.edu}

\address{Alfredo Hubard, Universit\'e Gustave Eiffel, LIGM. UMR 8049, CNRS, ENPC, ESIEE, UPEM, F-77454, Marne-la-Vall\'ee, France.}
\email{alfredo.hubard@univ-eiffel.fr}

\address{Roman Karasev, Institute for Information Transmission Problems RAS, Bolshoy Karetny per. 19, Moscow, Russia 127994}
\email{r\_n\_karasev@mail.ru}
\urladdr{http://www.rkarasev.ru/en/}

\begin{document}

\begin{abstract}
Inspired by the classical Riemannian systolic inequality of Gromov we present a combinatorial analogue providing a lower bound on the number of vertices of a simplicial complex in terms of its edge-path systole. Similarly to the Riemannian case, where the inequality holds under a topological assumption of ``essentiality'', our proofs rely on a combinatorial analogue of that assumption. Under a stronger assumption, expressed in terms of cohomology cup-length, we improve our results quantitatively. We also illustrate our methods in the continuous setting, generalizing and improving quantitatively the Minkowski principle of Balacheff and Karam; a corollary of this result is the extension of the Guth--Nakamura cup-length systolic bound from manifolds to complexes.
%which improves the bound of Nabutovsky by an exponential factor and generalize the Guth--Nakamura bound from manifolds to complexes. Finally we generalize and improve quantitatively the Minkowski principle of Balacheff and Karam.
\end{abstract}

\maketitle

\section{Introduction}

What is the smallest number of vertices in a simplicial complex triangulating a given topological space? Taking $\R P^n$ as an example, the exact minimum is known only for $n\le4$. Asymptotically there is a wide gap between the best known lower bound of $\frac{1}{2}(n+1)(n+2)$ from \cite[\S\,16]{arnoux-marin1991} and the recently discovered \cite{adiprasito2020} upper bound of $e^{C\sqrt{n}\log n}$, which is realized as the quotient of a centrally symmetric simplicial convex polytope. \\

%The real projective space $\R P^n$ is just one instance of a so-called $n$-essential space (see below).

For any simplicial complex, the length of the shortest non-contractible loop along the edges is at least $3$. Hence the question of vertex minimal triangulations is the end-case of a question in systolic geometry.\\
%The double cover of any triangulation of projective space is a simplicial sphere with a free involution such that any two ``antipodal'' vertices of the cover must be at least $3$ edges apart.

In this paper we address the smallest number of vertices of triangulations such that the shortest non contractible loop along the edges has length at least $s$. One of our main results, Theorem \ref{theorem:a-la-gromov} below provides a combinatorial analogue of Gromov's systolic inequality in which the volume of a Riemannian metric is substituted by the number of vertices of a simplicial complex. This theorem answers a question from \cite[Appendix~1]{cdv-h-dm2015} and is valid not only for the projective space, but much more generally, for any combinatorially $n$-essential complex.\\

Besides that, in Theorem \ref{theorem:nabutovsky} below we address the classical systolic inequality for Riemannian polyhedra, improving the results of \cite{guth2010,nakamura2013,balacheff2016,borghini2019}.

\subsection{Discrete formulation of systolic inequalities}

\emph{Systolic} (or \emph{isosystolic}) inequalities, first studied by L\"owner and Pu~\cite{pu1952} (see also \cite{ivanov2011} for a modern version of Pu's two-dimensional result), relate the volume of a closed manifold to the \emph{systole}, that is, the length of a shortest non-contractible curve. Gromov \cite{gromov1983} proved \emph{the} systolic inequality: there exists a constant $c_n$ such that  if $M^n$ is an \emph{essential} manifold, then for any Riemannian metric $g$ on $M$ the systole $\sys(g)$ of $(M,g)$ is bounded from above by $c_n \vol(g)^{1/n}$. Here the constant $c_n$ depends only on the dimension, and $\vol(g)$ is the Riemannian volume of $(M, g)$. The converse---every non-essential manifold admits a metric with large systole but small volume---is also true, as first observed by Babenko~\cite{babenko1993}.

%For each manifold $M$, one wonders what is the smallest constant $c_M$ such that \[\sys(g)\leq c_M \vol(g)^{1/n}\] for all Riemannian metrics. Whenever $c_M>0$ we say that  $M$ is said to satisfy a Riemannian systolic inequality. Gromov showed that \emph{essential} manifolds satisfy a Riemannian systolic inequality. It is not hard to show that manifolds that satisfy the systolic inequality are essential,

\begin{definition}
The \emph{(edge-path) systole} $\sys X$ of a simplicial complex $X$ is the smallest integer such that any closed path along the edges of $X$ of edge-length less than $\sys X$ is null-homotopic. If every component of $X$ is simply connected, the systole is, by convention, $+\infty$.
\end{definition}

In \cite{cdv-h-dm2015} and \cite{kowalick2015} it was observed that there exists a constant $c_n'$ such that any simplicial complex $X$ triangulating an essential manifold has at least $c_n' \sys(X)^{n}$ facets (faces of dimension $n$). Theorem \ref{theorem:a-la-gromov} provides a condition on a simplical complex that implies the discrete systolic inequality in terms of the vertices (which in turn, implies the Riemannian version \cite{cdv-h-dm2015} and lower bounds for the number of faces of any given dimension
%all the entries of the $f$-vector 
which in particular improve significantly the estimate on the constant $c_n'$). 

Unlike the estimates on the number of $n$-faces, lower bounds on the number of vertices (in Theorem \ref{theorem:a-la-gromov}) are not easy to derive directly from Riemannian systolic inequalities. Instead our proof adapts the approach of Larry Guth~\cite{guth2010} and Panos Papasoglu~\cite{papasoglu2019}, which is metaphorically referred to as ``the Schoen--Yau~\cite{schoen1979} minimal hypersurface method''. 

We remark that metric systolic inequalities have been successfully studied in spaces other than manifolds; see, e.g., \cite[Appendix B]{gromov1983} or \cite{liokumovich2019}. Here we consider arbitrary simplicial complexes, which is not new, but allows to make definitions and present proofs in a very simple way.

Now we get down to definitions, paralleling the Riemannian definitions in the combinatorial setting. For a subset $W\subset V(X)$ of vertices of a simplicial complex $X$, we denote by $\langle W\rangle$ the subcomplex induced by $W$, that is, the subcomplex of all faces $\sigma\in X$ such that the vertices of $\sigma$ are in $W$.

\begin{definition}
A subset $Y\subset V(X)$ is called \emph{inessential} if the natural map $\pi_1(C)\to \pi_1(D)$ is trivial for every connected component $C$ of $\langle Y\rangle$ and every connected component $D$ of $X$.
\end{definition}

\begin{definition}
\label{definition:combinatorially-n-essential}
A complex $X$ is called \emph{combinatorially $n$-essential} if its vertex set cannot be partitioned into $n$ inessential sets or fewer.
\end{definition}

Note that in our definition 
\begin{itemize}
\item
Combinatorially $n$-essential implies combinatorially $m$-essential whenever $m\le n$.
\item 
Combinatorially $1$-essential is equivalent to having non-contractible loops and having finite systole.
\item 
A space having no non-contractible loops may be considered combinatorially $0$-essential, but our results do not apply to this case.
\end{itemize} 

\begin{theorem}
\label{theorem:a-la-gromov}
Let $X$ be a combinatorially $n$-essential simplicial complex, $n\ge 1$. Then the number of vertices of $X$ is at least
\[
%{n+\left\lfloor\frac{\sys X}{2} \right\rfloor \choose n} \geq \frac{1}{n!} \left\lceil\frac{\sys X}{2} \right\rceil^n.
\binom{n + \left\lfloor\frac{\sys X}{2}\right\rfloor - 1}{n-1} + 2 \binom{n + \left\lfloor\frac{\sys X}{2}\right\rfloor - 1}{n} - 1 \ge \binom{n + \left\lfloor\frac{\sys X}{2} \right\rfloor}{n} \ge \frac{1}{n!} \left\lceil\frac{\sys X}{2} \right\rceil^n.
\]
\end{theorem}

A topological space is said to be \emph{$n$-essential} if the classifying map $f : X\to K(\pi_1(X), 1)$ cannot be deformed to the $(n-1)$-skeleton of $K(\pi_1(X), 1)$, see \cite[Appendix B]{gromov1983}.  In an appendix to this paper we show that any triangulation of an $n$-essential space is combinatorially $n$-essential. The converse is not true, Remark \ref{remark:complete-graph} presents a combinatorially $n$-essential complex that is not $n$-essential as a topological space. See the appendix for more details on the comparison of these definitions.

\subsection{Size of metric balls}

The proof of theorem \ref{theorem:a-la-gromov}  will be based on considering certain metric balls and relating their properties to the systole of the simplicial complex $X$.  For $x\in V(X)$ and an integer $i$, let $B(x,i)\subset V(X)$ denote the set of vertices of a simplicial complex $X$ whose edge-distance to $x$ is at most $i$. Similarly, let $S(x,i)\subset V(X)$ denote the set of vertices of $X$ whose edge-distance to $x$ is exactly $i$.

%\begin{definition}
%We say that $X$ has \emph{homotopy triviality radius $r$} if every metric ball $B(x,r) \subseteq V(X)$ is inessential, and $r$ is the maximum integer with this property.
%\end{definition}

\begin{proposition}
\label{claim:systole-inessential-triviality-relation}
Let $X$ be a simplicial complex with finite systole. Set $r := \left\lfloor\frac{\sys X}{2} \right\rfloor-1$. Then every metric ball $B(x,r) \subseteq V(X)$ is inessential, and $r$ is the maximum integer with this property.
\end{proposition}
\begin{definition}
We call $r$ in the previous proposition the \emph{homotopy triviality radius} of $X$.
\end{definition}
\begin{proof}[Proof of Proposition \ref{claim:systole-inessential-triviality-relation}]
Consider the universal covering map $\widetilde X\to X$. By the definition of systole, the edge-distance between any two vertices of the same $\pi_1(X)$-orbit is at least $\sys X$. It follows that for all $r$ such that $2r < \sys X -1$ the restriction of the covering map $\widetilde X\to X$ to the (preimage of the) metric ball $\langle B(x,r)\rangle$ is a trivial cover. This means that this ball is inessential.

On the other hand, if $x$ is chosen on a non-contractible edge-loop of length $\sys X$, then $\langle B(x,r)\rangle$, for $r := \left\lfloor\frac{\sys X}{2} \right\rfloor$, contains this loop.
\end{proof}

The estimate on the number of vertices of $X$ follows from the estimate of the number of vertices in a carefully chosen ball in $X$ of radius $r+1$, given in the following theorem:

\begin{theorem}
\label{theorem:base-case}
Let $X$ be a combinatorially $n$-essential simplicial complex, $n\ge 1$, and $r$ be the homotopy triviality radius of $X$. Then there exists a vertex $x\in X$ such that for any $i=0,\ldots,r+1$  the number of vertices in $B(x,i)$ is at least $b_n(i)$, where positive integers $b_n(i)$ satisfy the following recursive relations:
\begin{itemize}
\item $b_1(i) = 2i+1$ for any $i=0,\ldots,r$, and $b_1(r+1)= 2r+2$;
\item $b_n(i) = \sum_{0\le j\le i}b_{n-1}(j)$ for any $i=0,\ldots,r+1$.
\end{itemize}
In particular, $b_n(i) \ge 2\binom{i + n -1}{n} + \binom{i + n-1}{n-1}$ for any $i=0,\ldots,r$, and $b_n(r+1)\ge 2\binom{r + n}{n} + \binom{r + n}{n-1} - 1$.
\end{theorem}

These results are proved in Section~\ref{section:essential-systole}.

\subsection{Improvements under additional assumption of cohomology cup-length}

We improve the bounds of Theorem~\ref{theorem:base-case} assuming additionally that the property of being essential is implied by a long nonzero product in cohomology. From here on we denote by $\Z_2$ the ring of residues modulo $2$.

\begin{definition}
\label{definition:cup-essential}
We call a complex \emph{$n$-cup-essential} (over $\Z$ or $\Z_2$) if it admits $n$ (not necessarily distinct) cohomology classes in degree $1$ (with coefficients $\Z$ or $\Z_2$, respectively) whose cup-product is non-zero.
\end{definition}

A cup-essential complex is essential (with the same $n$), as it follows from the Lusternik--Schnirelmann-type of argument (see Lemma~\ref{lemma:MW}). The torus $(S^1)^n$ is $n$-cup-essential both over $\Z$ and $\Z_2$. The real projective space $\R P^n$ is $n$-cup-essential over $\Z_2$.

\begin{definition}
We say that $X$ has \emph{homology triviality radius $r$ with respect to the $1$-cohomology classes} $\xi_i$ if the restriction of every $\xi_i$ to every metric ball $\langle B(x,r) \rangle \subseteq X$ is zero, and $r$ is the maximum integer with this property.
\end{definition}

Note that the existence of non-trivial degree $1$ cohomology classes implies the existence of non-contractible loops and allows to define the finite systole. 

\begin{theorem}
\label{theorem:cup-case}
Let $X$ be a simplicial complex that is $n$-cup-essential over $\Z$ or $\Z_2$, as witnessed by the degree $1$ cohomology classes $\xi_1,\ldots,\xi_n$, $\xi_1\smile\dots\smile\xi_n \neq 0$. Let $r$ be the corresponding homology triviality radius.

Then there is a vertex $x\in X$ such that for any $i=0,\ldots,r+1$  the number of vertices in $B(x,i)$ is at least $\b_n(i)$, where positive integers $\b_n(i)$ satisfy the following recursive relations:
\begin{itemize}
\item $\b_1(i) = 2i+1$ for any $i=0,\ldots,r$;% and $\b_1(r+1)= 2r+2$.
%\item $\b_n(r+1) = \b_{n-1}(r+1) + \b_{n-1}(r) + 2\sum_{0\le j\le r-1}\b_{n-1}(j)$.
\item $\b_n(i) = \b_{n-1}(i) + 2\sum_{0\le j\le i-1}\b_{n-1}(j)$ for any $i=0,\ldots,r$.
\end{itemize}
\end{theorem}

In the case of $\Z_2$ coefficients, this theorem is basically a discrete analogue of \cite[Theorem~2.3]{nakamura2013} with the improvement that $X$ is not required to be a manifold. In the case of $\Z$ coefficients, to our knowledge, neither discrete nor continuous version was known before, and we discuss the corresponding continuous statement below. 

\begin{remark}
\label{remark:generating}
Unraveling the recursion we obtain the generating function:
\[
\sum_{n,i} \b_n(i) u^n v^i = \sum_n \frac{u^n (1+v)^n }{(1-v)^{n+1}} = \frac{1}{1-v}\frac{1}{1 - u\frac{1+v}{1-v}} = 
\frac{1}{1-v-u -uv}.
\]
From the first identity in this formula and the expansion $\frac{1}{(1+v)^{n+1}} = \sum_i \binom{i+n}{n} v^i$ one obtains an estimate
\[
2^n \binom{i}{n} \le \b_n(i) \le 2^n \binom{i+n}{n}.
\]
For example, for a fixed $n$ and $i\to\infty$ we obtain
\[
\b(n,i) = 2^n \frac{i^n}{n!} (1 + o(1)).
\]
\end{remark}

For the total number of points in a cup-essential manifold we have:

\begin{theorem}
\label{theorem:a-la-gromov-cup}
Let $X$ be a simplicial complex that is $n$-cup-essential over $\Z$ or $\Z_2$. Then the number of vertices of $X$ is at least
\[
\sum_{k=0}^n \b_k\left(\left\lfloor\frac{\sys X}{2} \right\rfloor-1\right) \ge 2^n \binom{\left\lfloor\frac{\sys X}{2} \right\rfloor-1}{n}.
\]
\end{theorem}

If we are interested in face numbers of positive dimension, not only vertices, then Theorem \ref{theorem:a-la-gromov-cup} can be combined with lower bound theorems from face number theory.  Recall that $f_k$ denotes the number of $k$-dimensional faces. For example Kalai's lower bound theorem for simplicial PL-manifolds \cite{kalai1987} provides a lower bound of  ${n-1 \choose k} f_0-\binom{n}{k+1}k$ for the number of $k$ faces, for $1<k<n$ and roughly $nf_0$ for $k=n$ or $k=1$. Here is another interesting application related to our motivating example:

\begin{corollary} If a centrally symmetric polytope in $\mathbb R^{n+1}$ has edge-distance between any pair of antipodal vertices at least $s$, then for the $f$-vector the following inequalities hold:
\begin{enumerate}
\item 
$f_0\ge 2^{n-1} \binom{\left\lfloor s/2 \right\rfloor}{n}$;
%\frac{1}{2}\frac{s^n}{n!}$;
\item 
$f_k\ge  \binom{n+1}{k}(2^{n} \binom{\left\lfloor s/2 \right\rfloor}{n}-2n)+ 2^{k+1}\binom{n+1}{k+1}$;
\item 
$f_{n-1}\ge 2^n n\binom{\left\lfloor s/2 \right\rfloor}{n} +2^{n+1}-2n^2+2n+4$.
\end{enumerate}
\end{corollary}

The first item follows by applying Theorem \ref{theorem:a-la-gromov-cup} to the quotient of the polytope by the involution $x\mapsto -x$, for $s\ge 3$, which is an $n$-cup-essential triangulated $\mathbb RP^{n-1}$. The second and third items are straightforward applications of Stanley's results \cite{stanley1987}.  In the end-case of this corollary $s=2$ notice that any centrally symmetric polytope has at least $2n$ vertices. 

The following theorem shows that when $\sys X \gg n$, Theorems \ref{theorem:a-la-gromov} and \ref{theorem:a-la-gromov-cup} and the previous corollary are asymptotically optimal.

\begin{theorem}
\label{theorem:upper-bound}
For every $n$ and $s\ge 3$, there exists a centrally symmetric 
triangulation of the $n$-dimensional sphere $\widetilde X$
%convex polytope $P\subset \mathbb R^{n+1}$, 
such that the quotient $n$-projective space $X=\widetilde X/\Z_2$ has $\sys X=s$ and no more than $s^n$ vertices.
\end{theorem}

\begin{remark}
In seems to us that the estimates on the number of vertices of Theorems \ref{theorem:a-la-gromov} and \ref{theorem:a-la-gromov-cup} are valid, more generally, for $\Delta$-complexes in place of simplicial complexes \cite[Section 2.1]{hatcher2001} with the same proofs. Of course, Theorem \ref{theorem:any-systole} below does not extend since $\Delta$-complexes may have systole equal to $1<3$.
\end{remark}

The proofs of the vertex number estimates under the cup-length assumption are given in Section \ref{section:cohomology-systole}.

\subsection{No assumption on the systole}

The question of estimating the number of vertices in a triangulation of a manifold, mentioned in the beginning, may be viewed as a particular case of the systolic problem. Every combinatorially $n$-essential complex $X$ has $\sys X\ge 3$, just because any two edges with the same endpoints must be equal. In this case our theorems only provide a lower bound for $|V(X)|$ that is linear in $n$. In remark \ref{remark:complete-graph} below, we show that this cannot be improved. 

But this linear bound may be improved under stronger assumptions. This was noticed in \cite[\S16]{arnoux-marin1991}, see also a series of somewhat more general estimates in \cite{govc2020-1} and \cite[Section 3]{govc2020-2}. Here we provide an analogue of these results in terms of combinatorial essentiality.

\begin{theorem}
\label{theorem:any-systole}
If the barycentric subdivision of simplicial complex $X$ is combinatorially $n$-essential then $|V(X)|\ge (n+1)(n+2)/2$.
\end{theorem}

\begin{remark}
\label{remark:complete-graph}
Considering a barycentric subdivision (as a step towards the topological definition of $n$-essential from the combinatorial one) in this theorem is important. The complete graph $K_{2n+1}$ is combinatorially $n$-essential (partitioning its vertices into $n$ or fewer parts results in a triangle) and has a linear (in terms of $n$) number of vertices. In particular, \emph{one cannot even guarantee that a combinatorially $n$-essential complex has faces of dimension $n$.} See more on comparison of the notions of $n$-essential in the Appendix.
\end{remark}

The lower bound in Theorem \ref{theorem:any-systole} should be compared with the upper bound $e^{C\sqrt{n}\log n}$ in the construction of a centrally symmetric simplicial convex polytope  in \cite{adiprasito2020}. It is wide open either to improve this quadratic lower bound, or to give a construction of an $n$-essential simplicial complex with a quadratic (or merely polynomial) in $n$ number of vertices.

\subsection{Improvements to the Riemannian inequalities}

The ideas behind Theorem~\ref{theorem:cup-case} can be used to improve the dimension-dependent factor in the classical systolic inequality, if one additionally assumes that the space is cup-essential (see Definition~\ref{definition:cup-essential}). This version of the systolic inequality for manifolds first appeared (over $\Z_2$) in the work of Guth~\cite{guth2010}, which pioneered the use of the minimal hypersurface method in systolic geometry. Our result applies to a wider class of complexes (Riemannian polyhedra), works over $\Z$ as well, and the dimensional factor is as good as the best known factor for cup-essential manifolds~\cite{nakamura2013}.

A \emph{Riemannian polyhedron} is a simplicial complex whose simplices are endowed with Riemannian metrics matching on the intersection of adjacent simplices. This is a reasonable class of path metric spaces including Riemannian manifolds and admitting systolic inequalities. The systole $\sys X$ of a compact non-simply connected Riemannian polyhedron $X$ is the length of a shortest non-contractible loop in $X$ (not necessarily along the edges of the simplicial structure).
%A polyhedron of dimension $n$ is said to be \emph{essential} if it is $n$-essential in the sense of Lemma~\ref{lemma:essential}.

Following the approach of Papasoglu~\cite{papasoglu2019}, Nabutovsky~\cite{nabutovsky2019} gave a transparent proof of the systolic inequality $\sys X \le 2 (n!/2)^{1/n} \vol(X)^{1/n}$ for essential Riemannian $n$-polyhedra. The factor $2 (n!/2)^{1/n}$ is the best known. We improve Nabutovsky's estimate by a factor of $2^{1 - 1/n}$, assuming a cohomological cup-length condition. In fact, we prove more.

%\begin{definition}
%We call a complex of dimension $n$ \emph{cup-essential} (over $\Z$ or $\Z_2$) if it admits $n$ (not necessarily distinct) cohomology classes in degree $1$ (with coefficients $\Z$ or $\Z_2$, respectively) whose cup-product is non-zero.
%\end{definition}
%Note that a cup-essential space is essential in the sense of Lemma~\ref{lemma:essential}. The torus $(S^1)^n$ is cup-essential both over $\Z$ and $\Z_2$. The real projective space $\R P^n$ is cup-essential over $\Z_2$.

For a cohomology class $\xi \in H^1(X; \Z)$ or $\xi \in H^1(X; \Z_2)$ we define $L(\xi)$ as the shortest length of a loop in $X$ on which $\xi$ evaluates non-trivially. The \emph{Minkowski principle} of Balacheff--Karam~\cite{balacheff2016} (see also~\cite{balacheff2021}) states that if $M$ is a closed $n$-manifold that is $n$-cup-essential over $\Z_2$ with $\xi_1\smile\dots\smile\xi_n\not=0$, then
\[
\prod_{i=1}^{n} L(\xi_i) \le c_n \vol_n X.
\]
The name ``Minkowski principle'' refers to Minkowski's second theorem on centrally symmetric convex bodies and lattices, which can be regarded as a version of this inequality for a flat Finsler torus.

The following result improves or generalizes \cite[Theorem~3]{guth2010}, \cite[Theorem~2.3]{nakamura2013}, \cite[Theorem~2]{balacheff2016}, \cite[Theorem~1.1]{borghini2019}.
%The scheme of the proof combines the ideas from the proof of Theorem \ref{theorem:integer-case} with the ones from the work of Nabutovsky.

\begin{theorem}
\label{theorem:nabutovsky}
Let $X$ be a compact Riemannian polyhedron of dimension $n$. Suppose it is $n$-cup-essential over $\Z$ or $\Z_2$, which is witnessed by the degree $1$ cohomology classes $\xi_1,\ldots,\xi_n\in H^1$, $\xi_1\smile\dots\smile\xi_n \neq 0$, arranged so that $L(\xi_1) \le \ldots \le L(\xi_n)$. Then there exists a point $x \in X$ such that the metric balls around $x$ satisfy the following volumetric bounds:

\[
\vol_n B(x, r) \ge
\left\{
\begin{array}{ll}
\frac{(2r)^n}{n!}, & 0 < r \le \frac{L(\xi_1)}{2}, \\
\frac{L(\xi_1)(2r)^{n-1}}{n!}, & \frac{L(\xi_1)}{2} < r \le \frac{L(\xi_2)}{2}, \\
\cdots, & \cdots \\
\frac{L(\xi_1) \ldots L(\xi_{n-2})(2r)^{2}}{n!}, & \frac{L(\xi_{n-2})}{2} < r \le \frac{L(\xi_{n-1})}{2}, \\
\frac{L(\xi_1) \ldots L(\xi_{n-1}) 2r}{n!}, & \frac{L(\xi_{n-1})}{2} < r \le \frac{L(\xi_n)}{2}.
\end{array}
\right.
\]

In particular, the Minkowski principle holds:
\[
\prod_{i=1}^{n} L(\xi_i) \le n! \vol_n X.
\]
Even more specifically, $\sys X \le (n!)^{1/n} \vol_n (X)^{1/n}$.
\end{theorem}

The proof is given in Section~\ref{section:riemannian-systole}.

%\subsection{Paper structure}

%Section~\ref{section:essential-systole} contains the proof of Theorem~\ref{theorem:base-case}, Theorem~\ref{}
%In the subsequent sections we prove the results mentioned in the introduction. In Section \ref{section:comparison} we discuss how our definition of ``combinatorially $n$-essential'' relates to the topological definitions of ``$n$-essential''.

\subsection*{Acknowledgments}
The authors thank Alexander Kamal and the unknown referee for pointing out some typos, miscalculations, and simplifications.

%%%%%%%%%%%%%%%%%%%%%%%%%%%%%%%%%%%%%%%%%%%%%%%%%%%%%%%%%%%%%%%%%%%%%%%%%%%%%%%%%%%%%%%%%%%
\section{Using the combinatorially essential assumption}
\label{section:essential-systole}

In order to apply induction, we generalize Theorem \ref{theorem:base-case} and then prove the generalization. We first generalize the notion of  combinatorially inessential and combinatorially essential to arbitrary covers.

\begin{definition}
Assume a simplicial complex $X$ has a covering map $\pi : \widetilde X\to X$. A subset $Y\subset V(X)$ is \emph{$\pi$-inessential} if the restriction $\pi^{-1}(\langle Y\rangle) \to \langle Y\rangle$ is a trivial cover. We assume that a cover over $\langle Y \rangle$ is trivial if and only if it is trivial (equal to $C\times D$ with a discrete set $D$) over every connected component $C$ of $\langle Y \rangle$.
\end{definition}

\begin{definition}
\label{definition:n-ess-cover}
Assume a simplicial complex $X$ has a covering map $\pi : \widetilde X\to X$. This covering map is called combinatorially $n$-essential if the vertex set of $X$ cannot be partitioned into $n$ or fewer $\pi$-inessential sets.
\end{definition}

\begin{definition}
We say that a covering map $\pi : \widetilde X\to X$ has \emph{homotopy triviality radius $r$} if every metric ball $B(x,r) \subseteq V(X)$ is $\pi$-inessential and $r$ is the maximum integer with this property.
\end{definition}

The definitions stated in the introduction are obtained from these definitions when $\widetilde X\to X$ is the universal covering map. When $X$ has several connected components we assume that $\widetilde X$ is the disjoint union of the universal covers of every component.

\begin{theorem}
\label{theorem:base-case-gen}
Let a simplicial complex $X$ have a covering map $\pi : \widetilde X\to X$ and let $\pi$ be combinatorially $n$-essential. Let $r$ be the homotopy triviality radius of $\pi$. Then there exists a vertex $x\in X$ such that for any $i=0,\ldots,r+1$  the number of vertices in $B(x,i)$ is at least $b_n(i)$, where positive integers $b_n(i)$ satisfy the following recursive relations:
\begin{itemize}
\item $b_1(i) = 2i+1$ for any $i=0,\ldots,r$ and $b_1(r+1)= 2r+2$;
\item $b_n(i) = \sum_{0\le j\le i}b_{n-1}(j)$ for any $i=0,\ldots,r+1$.
\end{itemize}
%In particular, $b_n(i) \ge {n+i \choose i}$ for any $i=0,\ldots,r+1$.
In particular, $b_n(i) \ge 2\binom{i + n -1}{n} + \binom{i + n-1}{n-1}$ for any $i=0,\ldots,r$, and $b_n(r+1)\ge 2\binom{r + n}{n} + \binom{r + n}{n-1} - 1$.
\end{theorem}

\begin{proof}[Proof of Theorem~\ref{theorem:base-case} assuming Theorem~\ref{theorem:base-case-gen}]
We may take a single component of $X$ and consider the universal covering map $\pi : \widetilde X\to X$. Then $\pi$-inessential is the same as inessential for (the fundamental group of) $X$. Hence Theorem~\ref{theorem:base-case-gen} applies and produces the result.
\end{proof}

\begin{proof}[Proof of Theorem~\ref{theorem:base-case-gen}]
Let us first prove the theorem for $n=1$. In that case, consider a simple (passing once through each of its vertices) closed loop, whose lift to $\widetilde X$ is not closed. Such a loop exists since otherwise $X$ would be inessential. The loop has at least $2r+2$ vertices, because otherwise it would fit into a ball of radius $r$ containing only loops lifting to loops. The equalities $b_1(i) = 2i+1$ for any $i=0,\ldots,r$ and $b_1(r+1)= 2r+2$ then easily follow.

We proceed by induction over $n$. Assume that $n>1$ and that the theorem is already proven for smaller $n$. Find a smallest subset of vertices $Z\subset V(X)$ such that its complement $Y:=V(X)\setminus Z$ is $\pi$-inessential.

By definition, the restriction of the covering map, $\pi^{-1}(\langle Z\rangle)\to \langle Z\rangle$, is then combinatorially $(n-1)$-essential. The metric balls of $Z$ are contained in the respective metric balls of $X$ and the $\pi$-inessentiality property of the balls is therefore preserved. Hence the homotopy triviality radius of the restriction of the covering map is no smaller than that of $\pi$.

By the induction hypothesis, there exists a vertex $x\in Z$ such that for any $i=0,\ldots,r+1$ the number of vertices in $B(x,i)\cap Z$ is at least $b_{n-1}(i)$.

Consider the following modification of $Z$:
\[
Z':=Z\cup S(x,r+1)\setminus B(x, r).
\]
Let us prove that $Y':=V(X)\setminus Z'$ is $\pi$-inessential.

Assume to the contrary, that there is a simple closed edge-path $P$ in $\langle Y' \rangle$ whose lift to $\widetilde X$ is not closed. The set of vertices $S(x,r+1)$ separates the $1$-skeleton of $X$ into two disconnected components with vertex sets $B(x,r)$ and $V(X)\setminus B(x,r+1)$, respectively. Since $P$ is simple and does not contain vertices of $S(x,r+1)$, we have that
\begin{itemize}
\item
either $P\subset \langle B(x,r)\rangle$;
\item
or $P\subset \langle Y'\setminus B(x,r+1)\rangle$.
\end{itemize}
In the first case, $P$ lifts to a closed path in $\widetilde X$ because $\pi$ has the homotopy triviality radius $r$. In the second case, $P$ lifts to a closed path in $\widetilde X$ because $Y'\setminus B(x,r+1)\subseteq Y$ and $Y$ is $\pi$-inessential.

Since $Y'$ is $\pi$-inessential, from the minimality of $Z$, we obtain that
\[
|Z| \leq |Z'|,
\]
and therefore
\[
|S(x,r+1)\setminus Z| \ge |B(x,r)\cap Z|.
\]
Analogously, for every $i \leq r$ we obtain that
\[
|S(x,i+1)\setminus Z| \ge |B(x,i)\cap Z|.
\]

Finally, for any $i=0,\ldots,r+1$ we have that
\begin{multline}
\label{equation:ball-estimate}
|B(x,i)| = |B(x,i)\cap Z| + \sum_{0\le j\le i}|S(x,j)\setminus Z| \ge \\
\ge \sum_{0\le j\le i}|B(x,j)\cap Z| \ge \sum_{0\le j\le i} b_{n-1}(j)=b_n(i).
\end{multline}

It remains to prove the inequality $b_n(i) \ge 2\binom{i + n -1}{n} + \binom{i + n-1}{n-1}$. The case $n=1$ reads
\[
b_1(i) \ge 2i + 1,
\]
and has been already shown in the beginning of the proof.

For $n>1$ we have:
\begin{multline*}
b_n(i) = \sum_{0\le j\le i}b_{n-1}(j) \ge\\
\ge \sum_{0\le j\le i} \left( 2\binom{j + n - 2}{n-1} + \binom{j + n - 2}{n-2}\right) = 2\binom{i + n -1}{n} + \binom{i + n-1}{n-1}.
\end{multline*}

For $b_n(r+1)$, as compared to $b_n(i)$ with $i=0,\ldots,r$, the estimate is $1$ less in case $n=1$, as was described in the beginning of the proof. In the course of summation this $-1$ summand carries on, hence we have
\[
b_n( r+1) \ge 2\binom{r + n}{n} + \binom{r + n}{n-1} - 1.
\]
\end{proof}

\begin{proof}[Proof of Theorem \ref{theorem:a-la-gromov}]
From Proposition \ref{claim:systole-inessential-triviality-relation} we know that the restriction of the universal covering map $\pi : \widetilde X\to X$ to every metric ball $\langle B(x,r)\rangle$ of $X$ is inessential, where $r = \left\lfloor\frac{\sys X}{2} \right\rfloor - 1$.

By Theorem \ref{theorem:base-case} there exists a ball of radius $r+1$ in $X$ with at least 
\[
b_n(r+1)\ge 2\binom{r + n}{n} + \binom{r + n}{n-1} - 1
\]
vertices, which is all we need to prove.
%By Theorem \ref{theorem:base-case-gen} we find one such ball $B_n$ with at least $b_n(r)$ vertices. The restriction of $\pi$ to the complex $\langle V(X)\setminus B_n \rangle$ is combinatorially $(n-1)$-essential and by Theorem \ref{theorem:base-case-gen} we find a ball of this subcomplex with $b_{n-1}(r)$ vertices, and so on. Altogether, we exhibit at least
%\begin{multline*}
%\sum_{k=0}^n b_k(r) \ge 1 + \sum_{k=1}^n 2\binom{r + k - 1}{k} + \binom{r + k-1}{k-1} = -1 + \sum_{k=0}^n 2\binom{r + k - 1}{r-1} + \binom{r + k-1}{r} =\\
%= 2\binom{r + n}{r} + \binom{r + n}{r+1} - 1 = 2\binom{r + n}{n} + \binom{r + n}{n-1} - 1
%\end{multline*}
%vertices of $X$.
\end{proof}

\begin{proof}[Proof of Theorem \ref{theorem:any-systole}]
Start with the universal covering map $\pi : \widetilde X \to X$ and argue, as in the proof of Theorem \ref{theorem:base-case-gen}, considering a covering map and its restrictions.

We claim that the dimension of $X$ must be at least $n$. Otherwise, the vertices of its barycentric subdivision can be properly colored with $n$ colors, implying that $\pi$ is not combinatorially $n$-essential over the barycentric subdivision. Hence $X$ has a face $F\subset X$ with at least $n+1$ vertices. Remove this face from $X$ to obtain the complex $Y$ induced by $X$ on $V(X)\setminus F$.

The barycentric subdivision $Y'$ of $Y$ is contained in the barycentric subdivision $X'$ of $X$. The complex $Z$ induced on $V(X')\setminus V(Y')$ retracts onto $\langle F\rangle$, as it is contained in the neighborhood $U(F)$, and then retracts onto any point of $F$. Hence $Z$ is $\pi$-inessential. It follows that the restriction of the covering map $\pi$ to the preimage of $Y'$ is at least $(n-1)$-essential. This allows to make a step of induction, taking $Y$ in place of $X$ and keeping the covering map $\pi$ to measure essential and inessential. In the end we exhibit at least
\[
(n+1) + n + (n-1) + \dots + 1 = \frac{(n+1)(n+2)}{2}
\]
vertices of $X$.
\end{proof}

%%%%%%%%%%%%%%%%%%%%%%%%%%%%%%%%%%%%%%%%%%%%%%%%%%%%%%%%%%%%%%%%%%%%%%%%%%%%%%%%%%%%%%%%%%%
\section{Using the cup-essential assumption}
\label{section:cohomology-systole}

The following lemma (basically from the Lusternik--Schnirelmann theory) relates the $n$-essential property and the existence of a nonzero cohomology product of length $n$.

\begin{lemma}
\label{lemma:MW}
Let $X$ be a simplicial complex with its set of vertices $V(X)$ decomposed into a disjoint union $V(X)=Y \sqcup Z$. Suppose that there are two cohomology classes $\xi_1$ and $\xi_2$ on $X$ such that $\xi_1\smile\xi_2\neq 0$. Assume that $\xi_1$ restricted to $\langle Y\rangle$ is trivial.
Then $\xi_2$ restricted to $\langle Z\rangle$ is not trivial.
\end{lemma}
\begin{proof}
Let $U(Z)$ be the union of all open stars of the vertices in $Z$, $U(Z)\supset \langle Z\rangle$. And analogously $U(Y)\supset \langle Y\rangle$. The set $U(Z)$ retracts onto $\langle Z\rangle $ and $U(Y)$ retracts onto $\langle Y\rangle$. Since $U(Z)\cup U(Y) = X$,  $\xi_1\smile\xi_2\neq 0$, $\xi_1|_{U(Y)} = 0$, it follows from the properties of the cohomology multiplication (the same as used in the Lusternik--Schnirelmann category estimate) that $\xi_2|_{U(Z)} \neq 0$.
\end{proof}

\begin{proof}[Proof of Theorem~\ref{theorem:cup-case} over $\Z$]
We use induction. The base case $n=1$ follows from Theorem \ref{theorem:base-case} (or from the beginning of its proof).

The class $\xi_1$ corresponds to a classifying map
\[
X\to K(\mathbb Z, 1) = S^1 = \R/\Z.
\]
We pull back the covering map $\R\to S^1$ to a covering map $\pi: \widetilde{X} \to X$. This $\widetilde{X}$ has a free $\Z$-action with a bounded fundamental region $F$, $\widetilde{X} = \bigcup_{m\in\Z} mF$. We may also assume that $X$ and $\widetilde{X}$ are connected, or otherwise work in a connected component.

Let $Z\subset V(\widetilde{X})$ be a smallest set of vertices that ``separates $-\infty$ from $+\infty$'' in $\widetilde{X}$. That is, such set of vertices that for all sufficiently large $m$ any edge-path in $\widetilde{X}$ connecting a vertex in $mF$ to a vertex in $-mF$ contains a vertex in $Z$.

Note that we choose a finite $Z$ and therefore for sufficiently large positive integer $N$ the set $Z$ is not connected with any of its shifts by $kN$, for any nonzero $k\in\mathbb Z$. Consider $X_N := \widetilde{X}/N\mathbb Z$ and the natural projections $\pi_N : X_N \to X$ and $\pi'_N : \widetilde{X} \to X_N$. Choose $N$ sufficiently large and coprime with the order of $\xi_1\smile\dots\smile\xi_n$. Note that $\pi_N$ is an $N$-sheet cover and therefore $\pi_N^*(\xi_1\smile\dots\smile\xi_n) \neq 0$, since otherwise we would have
\[
N\xi_1\smile\dots\smile\xi_n = \tau(\pi_N^*(\xi_1\smile\dots\smile\xi_n)) = 0,
\]
where $\tau : H^*(X_N;\mathbb Z)\to H^*(X;\mathbb Z)$ denotes the transfer map.

The projection $\pi'_N$ induces an isomorphism of simplicial complexes $\langle Z\rangle \to \langle Z_N\rangle$ by the choice of $N$, where $Z_N := \pi'_N(Z)$. The subset $Z_N$ blocks any cycle in $X_N$ which could lift to a cycle connecting $-\infty$ and $+\infty$, as it follows from the separation property of $Z$. Therefore $\pi_N^*(\xi_1)$ is zero on $\langle V(X_N)\setminus Z_N \rangle$. By Lemma~\ref{lemma:MW}, $\pi_N^*(\xi_2)\smile\dots\smile\pi_N^*(\xi_n)$ is nonzero on $\langle Z_N\rangle$. From the isomorphism of simplicial complexes $\langle Z\rangle \cong \langle Z_N\rangle$, the product
\[
\pi^*(\xi_2)\smile\dots\smile\pi^*(\xi_n) \in H^{n-1}(Z;\mathbb Z)
\]
is also nonzero.

The projection of $B(x, r)\cap Z$ is contained in the corresponding $r$-ball $B(\pi(x),r)$ of $X$. By the choice of $r$, the classes $\pi^*(\xi_2),\ldots,\pi^*(\xi_n)$ vanish on $\langle B(x,r)\cap Z \rangle $. Hence the inductive assumption with $n$ replaced by $n-1$ applies to $Z$. Let $x\in Z$ be the point obtained from the inductive assumption.

Similar to the proof of Theorem \ref{theorem:base-case}, we have that for every $i+1 \leq r$
\[
|S(x,i+1)\setminus Z| \ge |B(x,i)\cap Z|.
\]
In the current setting this inequality can be improved.

Let $S^+(x,i+1)\subset S(x,i+1)\setminus Z$ be the subset of those vertices in $S(x,i+1)$ which are connected to $+\infty$, i.e., for which there exists an edge path in $\widetilde{X}$ not using any vertex of $Z$ and connecting them to a vertex in $mF$ for any sufficiently large $m$. Likewise, let $S^-(x,i+1)\subset S(x_1,i+1)\setminus Z$ be the subset of those vertices in $S(x,i+1)$ which are edge-connected to $-\infty$ in $\widetilde{X}\setminus Z$.

Remove $B(x, i)$ from $Z$ and add $S^+(x, i+1)$ instead. Denote the obtained vertex set by $Z'$. For $i+1\le r$, we have that $Z'$ still ``separates $-\infty$ from $+\infty$'' in $\widetilde{X}$. Indeed, any path from $-\infty$ to $+\infty$ in $\widetilde{X}$ not passing through $Z'$ has to pass through $Z\setminus Z' = Z\cap B(x, i)$. Then this path has to reach $+\infty$ from $B(x,i)$. When this path  passes through $S(x,i+1)$ for the last time, it does in fact pass through $S^+(x, i+1)$ by the definition of $S^+(x, i+1)$.

From the minimality of our choice of $Z$, we get that for any $i+1\le r$
\[
|S^+(x,i+1)| \ge |B(x,i)\cap Z|.
\]
Likewise,
\[
|S^-(x,i+1)| \ge |B(x,i)\cap Z|.
\]

We now apply the inductive assumption that bounds $|B(x,i)\cap Z|$ from below for all $i \leq r$ noting that each of these $B(x,i)$ projects bijectively onto $B(\pi(x), i)$ by the definition of the cohomology triviality radius. So, for all $i \leq r$ we have:

\begin{multline*}
%\label{equation:ball-estimate-z}
|B(\pi(x), i)| = |B(x,i)| = |B(x,i)\cap Z| + \sum_{0\le j\le i}|S(x,j)\setminus Z| \ge \\
\ge |B(x,i)\cap Z| + \sum_{0\le j\le i}|S^+(x,j)| + \sum_{0\le j\le i}|S^-(x,j)| \ge \\
\ge \b_{n-1}(i) + 2\sum_{0\le j\le i-1}\b_{n-1}(j)=\b_n(i).
\end{multline*}
\end{proof}

\begin{proof}[Proof of Theorem \ref{theorem:cup-case} over $\Z_2$]
We use induction. The base case $n=1$ follows from Theorem \ref{theorem:base-case} (or from the beginning of its proof).

The class $\xi_1$ corresponds to a classifying map
\[
X\to K(\Z_2, 1) = \mathbb RP^\infty.
\]
We pull back the double covering map $S^\infty \to \mathbb RP^\infty$ to a double covering map $\pi: \widetilde X \to X$. This $\widetilde X$ has a free involution $\tau : \widetilde X\to \widetilde X$ We may also assume that $X$ and $\widetilde X$ are connected, or otherwise work in a connected component.

Let $Z\subset V(X)$ be a smallest set of vertices such that the covering map $\pi$ is trivial on $\langle V(X)\setminus Z\rangle$. Then $\xi_1$ is zero over $\langle V(X)\setminus Z\rangle$ and the restriction of the product $\xi_2\smile\dots\smile\xi_n$ is nonzero over $\langle Z\rangle$ by Lemma \ref{lemma:MW}.

Any metric $r$-ball of $\langle Z\rangle$ is contained in the corresponding $r$-ball of $X$. By the cohomology triviality radius assumption on $X$, we have that the classes $\xi_2,\ldots,\xi_n$ vanish on any $r$-ball of $\langle Z\rangle$. Hence the inductive assumption with $n$ replaced by $n-1$ applies to $\langle Z\rangle$. Let $x\in Z$ be the point obtained from the inductive assumption and put $Y=V(X)\setminus Z$.

Similar to the proof of Theorem \ref{theorem:base-case}, we have that for every $i+1 \leq r$
\[
|S(x,i+1)\setminus Z| \ge |B(x,i)\cap Z|,
\]
but we are going to improve this bound.

Let $\widetilde Z$ be the lift of $Z$ to $\widetilde X$. Let the lift of $\langle S(x,i+1)\rangle$ to $\widetilde{X}$ consists of two disconnected copies of $\langle S(x,i+1)\rangle$ from the cohomology triviality assumption, let them be $\widetilde S$ and $\tau\widetilde S$. The same applies to the lift of $\langle Y\rangle$, it consists of two disconnected copies $\langle \widetilde Y\rangle, \langle\tau\widetilde Y\rangle\subset V(\widetilde X)$. Put
\[
\widetilde S^+ = (\widetilde S \cap \widetilde Y) \cup (\tau\widetilde S\cap \tau\widetilde Y).
\]
The components of $\langle \widetilde S^+\rangle$ in the above union formula are disconnected and therefore $\langle \widetilde S^+\rangle$ is a trivial double cover over $\langle S^+\rangle$, for the corresponding subset $S^+\subset S(x,i+1)\cap Y$. Analogously,
\[
\widetilde S^- = (\widetilde S \cap \tau\widetilde Y) \cup (\tau\widetilde S\cap \widetilde Y)
\]
generates $\langle S^-\rangle$, which is a trivial double cover over $\langle S^-\rangle$ for the corresponding $S^-\subset S(x,i+1)\cap Y$. Note that $S^+\cap S^- =\emptyset$.

Remove $B(x, i)$ from $Z$ and add $S^+$ instead. Denote the obtained vertex set by $Z'$. The new complement $Y'$ then equals $Y\cup B(x,i)\setminus S^+$. We need to show that $Z'$ and $Y'$ may serve as $Z$ and $Y$. For this, it is sufficient to assume that $\xi_1$ evaluates to $1$ on a loop $\gamma$ passing through the vertices of $Y'$ and obtain a contradiction.

If $\gamma$ does not touch $B(x, i)$ then it is fully contained in $Y$ and the contradiction is obtained by the choice of $Y$. Otherwise let $\gamma$ start and end in $B(x,i)$. Then its lift $\widetilde\gamma$ starts in $\widetilde B$ and ends in $\tau\widetilde B$, where $\widetilde B$ is a component of the lift of $B(x,i)$ to $\widetilde X$ corresponding to $\widetilde S$.

The path $\widetilde \gamma$ will need to leave $\widetilde B\cup \widetilde S$ through $\widetilde S^-$ (since $\gamma$ did not touch $S^+$) and hence it gets into $\tau\widetilde Y$ after this by the definition of $\widetilde S^-$. At some point it has to touch $\tau\widetilde S$ and, since it passes in $\tau\widetilde Y$, this will happen in $\widetilde S^+$. Hence $\gamma$ touches $S^+$, a contradiction.

From the minimality of our choice of $Z$, we now get that for any $i+1\le r$
\[
|S^+| \ge |B(x,i)\cap Z|.
\]
Similarly,
\[
|S^-| \ge |B(x,i)\cap Z|,
\]
and in total
\[
|S(x,i+1)\setminus Z| \ge 2|B(x,i)\cap Z|,
\]

We now apply the inductive assumption that bounds $|B(x,i)\cap Z|$ from below for all $i \leq r$ and obtain for any $i \leq r$:

\begin{multline*}
%\label{equation:ball-estimate-z}
|B(x,i)| = |B(x,i)\cap Z| + \sum_{0\le j\le i}|S(x,j)\setminus Z| \ge \\
\ge |B(x,i)\cap Z| + 2\sum_{0\le j\le i-1}|B(x,j)\cap Z|
\ge \b_{n-1}(i) + 2\sum_{0\le j\le i-1}\b_{n-1}(j)=\b_n(i).
\end{multline*}
\end{proof}

\begin{proof}[Proof of Theorem \ref{theorem:a-la-gromov-cup}]
Put $r = \left\lfloor\frac{\sys X}{2} \right\rfloor - 1$. The restriction of any $\xi_i$ to a metric ball of $X$ of radius $r$ is inessential, since such balls only contain contractible loops.

By Theorem \ref{theorem:cup-case} we find one such ball $B_n$ with at least $\b_n(r)$ vertices. The restriction of $\xi_1\smile\dots\smile\xi_{n-1}$ to the complex $\langle V(X)\setminus B_n \rangle$ is nonzero by Lemma \ref{lemma:MW}. By Theorem \ref{theorem:base-case-gen}, we find a ball of this subcomplex with $b_{n-1}(r)$ vertices, and so on. Altogether, we exhibit at least
\[
\sum_{k=0}^n \b_k(r) \ge \sum_{k=0}^n 2^k \binom{r}{k} \ge 2^n \binom{r}{n} = 2^n \binom{\left\lfloor\frac{\sys X}{2} \right\rfloor - 1}{n}.
\]
vertices of $X$.
\end{proof}

\begin{proof}[Proof of Theorem \ref{theorem:upper-bound}]
For every $n$, we construct a $\mathbb Z_2$-symmetric triangulation $X_n$ of the sphere $S^n$ with no more than $2s^n$ vertices and with the edge-distance between the opposite vertices at least $s$. The $\mathbb Z_2$-quotient of this complex has the desired properties.

We proceed by induction over $n$. For $n=1$, let $X_1$ be the polygon with $2s$ vertices, it evidently does the job.

Suppose we have already constructed $X_n$. Consider $(s-1)$ copies of $X_n$ which we call \emph{layers}. For each $i$ add a cylinder $X_n\times [0,1]$ with $X_n\times \{0\}$ and $X_n\times \{1\}$ identified with the $i\textsuperscript{th}$ and the $(i+1)\textsuperscript{th}$ layers respectively. Add cones over the first and the last layer with apexes $S$ and $N$ respectively. The resulting space is homeomorphic to $S^{n+1}$. The $\Z_2$ symmetry sends a vertex $v$ in the $i\textsuperscript{th}$ layer to the vertex $-v$ in the $(s-i)\textsuperscript{th}$ layer, where $-v$ is the vertex symmetric to $v$ in $X_n$. The north pole $N$ is sent to the south pole $S$ and vice versa.

It remains to triangulate the cylinders $X_n\times [0,1]$ between the layers. Assign a unique integer label to each vertex of $X_n$ so that the integers assigned to opposite vertices have the same absolute value but different sign. This way the vertices of $X_n$ are ordered according to their labels. For every top-dimensional simplex $\Delta\subset X_n$, triangulate $\Delta\times [0,1]$ in each cylinder according to the order of vertices of $\Delta$: the $i\textsuperscript{th}$ top-dimensional simplex of $\Delta\times [0,1]$ has the first $i$ vertices of $\Delta\times \{0\}$ and the last $(n+2-i)$ vertices of $\Delta\times \{1\}$ as its vertices. Denote the resulting triangulation by $X_{n+1}$.

It is easy to see that $X_{n+1}$ is symmetric and has not more than $(s-1)2s^n+2\leq 2s^{n+1}$ vertices. Its vertices are only connected by an edge to the vertices in the adjacent layers, and if a vertex $v$ in the $i\textsuperscript{th}$ layer is connected by an edge with a vertex $u$ in the $(i+1)\textsuperscript{th}$ layer then $v$ and $u$ were connected by an edge in $X_n$ as well. We obtain that
\begin{itemize}
\item
an edge-path between a vertex in some $i\textsuperscript{th}$ layer and its opposite in the $(s-i)\textsuperscript{th}$ layer, not passing through the poles, has length at least $s$;
\item
an edge-path between the poles has length at least $s$, because there are $(s-1)$ layers between the poles;
\item
an edge path between arbitrary vertices passing through a pole has length at least $s$, since the union of this path and its antipodal image is a closed path passing through both poles, whose length must be at least $2s$.
\end{itemize}
\end{proof}

\section{Minkowski principle for Riemannian polyhedra}
\label{section:riemannian-systole}

We prove inductively a certain technical version of Theorem~\ref{theorem:nabutovsky}. To state it, it is convenient to introduce the function showing up in the volume bound. For fixed numbers $0 \le L_1 \le \ldots \le L_n$, define the following monotone continuous function:
\[
V_n(r; L_1, \ldots, L_n) :=
\begin{cases}
\frac{(2r)^n}{n!} & \mbox{if } 0 < r \le L_1/2, \\
\frac{L_1(2r)^{n-1}}{n!} & \mbox{if } L_1/2 < r \le L_2/2, \\
\cdots & \\
\frac{L_1 \ldots L_{n-2}(2r)^{2}}{n!} & \mbox{if } L_{n-2}/2 < r \le L_{n-1}/2, \\
\frac{L_1 \ldots L_{n-1} 2r}{n!} & \mbox{if } L_{n-1}/2 < r \le L_n/2, \\
\frac{L_1 \ldots L_n}{n!} & \mbox{if } r > L_n/2.
\end{cases}
\]

\begin{lemma}
\label{lemma:technical}
For any $t\ge 0$ and any $0\le L_1\le L_2\le \dots \le L_n$ the following inequality holds
\[
2 \int_{0}^r V_{n-1}(t; L_1, \ldots, L_{n-1}) \;dt \ge V_n(r; L_1, \ldots, L_{n}).
\]
\end{lemma}

\begin{proof}
If $r \le L_1/2$,
\[
2 \int_{0}^r V_{n-1}(t; L_1, \ldots, L_{n-1}) \;dt = 2 \int_{0}^r \frac{(2t)^{n-1}}{(n-1)!} \; dt = \frac{(2r)^{n}}{n!} = V_n(r; L_1, \ldots, L_{n}).
\]

If $r \in (L_i/2, L_{i+1}/2]$, then inducting in $i$ we can assume that
\[
2 \int_{0}^{L_i/2} V_{n-1}(t; L_1, \ldots, L_{n-1}) \;dt \ge V_n(L_i/2; L_1, \ldots, L_{n}) = \frac{L_1 \ldots L_{i-1} L_i^{n-i+1}}{n!}.
\]
In order to get $2 \int_{0}^r V_{n-1}(t; L_1, \ldots, L_{n-1}) \;dt \ge V_n(r; L_1, \ldots, L_{n})$, we need the following inequality:
\[
 2 \int_{L_i/2}^r V_{n-1}(t; L_1, \ldots, L_{n-1}) \;dt \ge \frac{L_1 \ldots L_i ((2r)^{n-i} - L_i^{n-i})}{n!}.
\]
The left-hand side evaluates to
\[
 2 \int_{L_i/2}^r \frac{L_1 \ldots L_{i} (2t)^{n-i-1}}{(n-1)!} \; dt = \frac{L_1 \ldots L_i ((2r)^{n-i} - L_i^{n-i})}{(n-1)!(n-i)},
\]
thus proving the bound.

For $r > L_n/2$, the inequality follows from the fact that $V_n(r; L_1, \ldots, L_{n})$ remains constant.
\end{proof}

\begin{theorem}
\label{theorem:nabutovsky-technical}
Let $X$ be a compact Riemannian polyhedron of dimension $n$. Suppose it is $n$-cup-essential over $\Z$  or over $\Z_2$, which is witnessed by the degree $1$ cohomology classes $\xi_1,\ldots,\xi_n$, $\xi_1\smile\dots\smile\xi_n \neq 0$. Suppose that there are real numbers $0 < L_1 \le \dots \le L_n$ such that the class $\xi_i$ vanishes when restricted to a ball of radius less than $L_i/2$, for all $i$. Let $0 < \rho < L_1/2$, $0 < \epsilon < 1$ be any real numbers. Then there exists a point $x \in X$ such that $\vol_n B(x, r) \ge (1-\epsilon)V_n(r; L_1, \ldots, L_n)$ for all $r \ge \rho$.
\end{theorem}

\begin{proof}[Proof in the integral case]
If $n=1$, take a point $x$ on a shortest loop on which $\xi_1$ evaluates non-trivially. Then $B(x,r)$ intersects this loop along a curve of length at least $2r$ for $r \le L(\xi_1)/2$, and the statement follows.

Now assume $n>1$. For brevity, we assume fixed values of $L_1,\ldots, L_n$ and write
\[
V_n(r) = V_n(r; L_1, \ldots, L_n)\quad\text{and}\quad V_{n-1}(r) = V_{n-1}(r; L_1, \ldots, L_{n-1}).
\]

The class $\xi_n$ is classified by a map $X\to K(\mathbb Z, 1) = \R/\Z$, which gives rise to a covering map $\pi: \widetilde{X} \to X$.
The action of $\Z$ on $\widetilde{X}$ allows us to speak about ``separating $-\infty$ from $+\infty$'' in $\widetilde{X}$, as in the proof of Theorem~\ref{theorem:cup-case}.
%This $X_1$ has a free $\Z$-action with a bounded fundamental region $F$, $X_1 = \bigcup_{n\in\Z} nF$. We may also assume that $X$ and $X_1$ are connected, or otherwise work in a connected component.

Consider all bounded $(n-1)$-dimensional subpolyhedra\footnote{A \emph{subpolyhedron} is a subspace admitting the structure of a simplicial complex whose cells are embedded smoothly in the ambient polyhedron. The Riemannian metric is inherited from the ambient polyhedron, allowing one to measure intrinsic distances and volumes.} in $\widetilde{X}$, separating $-\infty$ from $+\infty$. For example, the boundary of any reasonable fundamental domain of the $\Z$-action is such a subpolyhedron.
Let $v$ be the infimum of their $(n-1)$-volumes. Pick such a subpolyhedron $Z$ with $\vol_{n-1} < v + \delta$, where $\delta>0$ is a small number to be specified later. Consider $X_N : = \widetilde{X}/N\mathbb Z$ where $N$ is a positive integer satisfying two properties:
\begin{itemize}
  \item $N$ is sufficiently large, so that the projection $\pi'_N : \widetilde{X} \to X_N$ induces an homeomorphism between $Z$ and $Z_N := \pi'_N(Z)$;
  \item $N$ is coprime with the order of $\xi_1\smile\dots\smile\xi_n$, so that its pullback under the projection $\pi_N : X_N \to X$ is non-zero.
\end{itemize}

Just like in the proof of Theorem~\ref{theorem:cup-case}, $\pi_N^*(\xi_n)$ vanishes on $X_N \setminus Z_N$. By the topological version of Lemma \ref{lemma:MW}, $\pi_N^*(\xi_1)\smile\dots\smile\pi_N^*(\xi_{n-1})$ restricts non-trivially to an arbitrarily small neighborhood of $Z_N$ in $X_N$; taking this neighborhood sufficiently small so that it retracts to $Z_N$, we see that $\pi_N^*(\xi_1)\smile\dots\smile\pi_N^*(\xi_{n-1})$ restricts non-trivially to $Z_N$.
Since $Z \simeq Z_N$, we have a long non-zero cohomology product in $H^*(Z;\mathbb Z)$, and we are in position to apply the induction hypothesis for $Z$.

Apply the induction hypothesis to $Z$ with modified $\hat\rho = \rho (\epsilon/2)^{1/n}$ and $\hat\epsilon = \epsilon/4$. It outputs a point $x \in Z$ such that $\vol_{n-1} B_{Z}(x, r) \ge (1-\hat\epsilon)V_{n-1}(r)$ for all $r \ge \hat\rho$. Here $B_{Z}(\cdot)$ denotes a metric ball in the intrinsic metric of $Z$, whereas the notation $B(\cdot)$ and $S(\cdot)$ will refer to $\widetilde{X}$. The intrinsic distances dominate the extrinsic ones, so we have
\[
\vol_{n-1} (B(x, r) \cap Z) \ge (1-\hat\epsilon)V_{n-1}(r) \quad \text{for all } r \ge \hat\rho.
\]

We now show that
\begin{equation}
\label{equation:star}
\vol_{n-1} S(x, r) \ge 2\left(1-\frac{\epsilon}{2}\right) V_{n-1}(r) \quad \text{for almost all } r \ge \hat\rho.\tag{$\star$}
\end{equation}

We only consider those $r$ for which $S(x, r)$ is a subpolyhedron. We can assume this is true for almost all $r$, perturbing the distance function slightly. Fix any such $r \ge \hat\rho$. Introduce $S^+ \subset S(x, r)$ (respectively, $S^-$) as the subset of points connected to $+\infty$ (respectively, $-\infty$) in $\widetilde{X} \setminus Z$. The set $Z \setminus (\mbox{int } B(x,r) \cap Z) \cup S^+$ still separates $-\infty$ from $+\infty$, hence we have a volume bound:
\[
v \le \vol_{n-1}\left( Z \setminus (\mbox{int } B(x,r) \cap Z) \cup S^+ \right) < v + \delta - (1-\hat\epsilon)V_{n-1}(r) + \vol_{n-1} S^+.
\]
Therefore, taking $\delta = \frac{\epsilon}{4} V_{n-1}(\rho)$, we obtain
\[
\vol_{n-1} S^+ > (1-\hat\epsilon)V_{n-1}(r) - \delta \ge \left(1-\frac{\epsilon}{2}\right)V_{n-1}(r).
\]
A similar bound can be established for $\vol_{n-1} S^-$, and adding those up we obtain~\eqref{equation:star}.

Now that we have \eqref{equation:star}, we can integrate it over $[\hat\rho, r]$ and use the coarea inequality:

\begin{multline*}
\vol_{n-1} B(x, r) \ge \int_{\hat\rho}^{r} \vol_{n-1} S(x, t) \; dt
\ge 2\left(1-\frac{\epsilon}{2}\right) \int_{\hat\rho}^{r} V_{n-1}(t) \; dt > \\
> 2\left(1-\frac{\epsilon}{2}\right) \int_{0}^{r} V_{n-1}(t) \; dt - 2 \int_{0}^{\hat\rho} V_{n-1}(t) \; dt \ge \\
\overset{\text{Lem.~\ref{lemma:technical}}}\ge \left(1-\frac{\epsilon}{2}\right) V_n(r) - \frac{(2 \hat\rho)^n}{n!}
\ge (1-\epsilon) V_n(r),
\end{multline*}
%\[
%\vol_{n-1} B(x, r) \ge \int_{\hat\rho}^{r} \vol_{n-1} S(x, t) \; dt \ge \left(1-\frac{\epsilon}{2}\right)\frac{2^n r^n}{n!} - \left(1-\frac{\epsilon}{2}\right)\frac{2^n \hat\rho^n}{n!} \ge (1-\epsilon)\frac{2^nr^n}{n!},
%\]
where the last inequality holds for all $r \ge \rho$ by the choice of $\hat\rho = \rho (\epsilon/2)^{1/n}$. For $r < L_n/2$, the projection $\pi$ sends $B(x, r)$ to $X$ isometrically, so we obtain the desired volume bound in $X$ for the radii in the range $[\rho, L_n/2]$. For $r \ge L_n/2$ the bound holds vacuously since $V_n(r)$ becomes constant.
\end{proof}

\begin{proof}[Proof in the mod $2$ case]
The strategy is the same as in the discrete case and the volumetric computations are the same as in the proof over $\Z$, so we only briefly sketch the argument.

The base case $n=1$ is done as before. Suppose that $n>1$ and that $X$ is connected. The class $\xi_n$ corresponds to a classifying map $X\to \mathbb RP^\infty$, along which we pull back the double covering map $S^\infty \to \mathbb RP^\infty$. This way we obtain a covering map $\pi: \widetilde X \to X$, with a free involution $\tau : \widetilde X\to \widetilde X$.

Consider all bounded $(n-1)$-dimensional subpolyhedra of $X$ such that $\pi$ restricts to a trivial cover over the complement $X\setminus Z$. Let $v$ be the infimal volume of such subpolyhedra. Pick a subpolyhedron $Z$ among those with the volume less than $v+\delta$, where $\delta$ is chosen as in the integral case. Just like in the discrete case, the product $\xi_1\smile\dots\smile\xi_{n-1}$ restricts non-trivially to $Z$, and one can apply the inductive assumption for $Z$, with parameters $\hat\rho$ and $\hat\epsilon$ chosen as in the integral case.

The induction hypothesis outputs a point $x \in Z$ such that $\vol_{n-1} (B(x, r) \cap Z) \ge (1-\hat\epsilon)V_{n-1}(r)$ for all $r \ge \hat\rho$. The surgery on $Z$ that replaces $B(x, r) \cap Z$ by $S(x, r)$ produces a subpolyhedron of volume at least $v$, and this gives a lower bound for $\vol_{n-1} S(x, r)$. To obtain a bound that is twice as good, we do the surgery more carefully. Similarly to the discrete case, one introduces non-overlapping subsets $S^+$ and $S^-$ of $S(x, r)$ in a way that makes $\xi_n$ vanish on the complement of $Z \setminus (\mbox{int } B(x,r) \cap Z) \cup S^+$ and on the complement of $Z \setminus (\mbox{int } B(x,r) \cap Z) \cup S^-$. The surgery on $Z$ that replaces $B(x, r) \cap Z$ by $S^+$ or $S^-$ produces a subpolyhedron of volume at least $v$, this gives lower bounds on $\vol_{n-1} S^+$ and $\vol_{n-1} S^-$, which we integrate and get the desired bound on $\vol_n B(x,r)$. The computations carry over from the integral case \emph{verbatim}.
\end{proof}

\begin{lemma}
\label{lemma:curve-factoring}
Let $\xi \in H^1(X; \Z)$ or $\xi \in H^1(X; \Z_2)$, and let $r < L(\xi)/2$. Then $\xi$ restricts to any ball of radius $r$ trivially.
\end{lemma}

\begin{proof}
Take an arbitrary loop $\gamma$ inside a ball of radius $r$. Introduce many points $p_i$ along the loop, and consider connect $x$ to $p_i$ by a geodesic segment $g_i$ of length at most $r$. For each $i$, consider the triangular curve $\gamma_i$ glued of $g_i$, $g_{i+1}$, and the short path between $p_i$ and $p_{i+1}$ along $\gamma$. If the $p_i$ are scattered densely enough, the length of each $\gamma_i$ is less than $L(\xi)$, hence $\xi(\gamma_i) = 0$. But when oriented appropriately, the $\gamma_i$ add up to $\gamma$; therefore, $\xi(\gamma) = 0$.
\end{proof}

\begin{proof}[Proof of Theorem~\ref{theorem:nabutovsky}]
It follows from Lemma~\ref{lemma:curve-factoring} that one can apply Theorem~\ref{theorem:nabutovsky-technical} with $L_1 = L(\xi_1), \ldots, L_n = L(\xi_n)$ (rearranging the $\xi_i$ if necessary).

%Set $R = \frac{\sys X}{2}$, and consider $0 < r< R$. The definition of the systole means that in the universal cover $\widetilde X$ the distance between any two points in the same $\pi_1(X)$-orbit is at least $\sys X$. Therefore, for any $\widetilde x\in \widetilde X$ the balls $B(g\widetilde x, r)$, for $g\in \pi_1(X)$, do not intersect and all project homeomorphically to $B(x,r)$, where $x$ is the image of $\widetilde x$ under the natural projection $\widetilde X\to X$. Hence, any $B(x,r)$ only contains contractible loops, and therefore the radius $R$ satisfies the cohomology assumption in the statement of Theorem \ref{theorem:nabutovsky-technical}.

We would like to set $\epsilon = 0$ and $\rho = 0$ in the conclusion of Theorem~\ref{theorem:nabutovsky-technical}. This will be done using the continuity of the function $v(x,r) := \frac{\vol_n B(x,r)}{V_n(r)}$.

First, we fix $\rho$ and get rid of $\epsilon$. Consider the function $v(x) = \min\limits_{r \in [\rho, R]} v(x,r)$. We know that for any $\epsilon$ there is $x$ such that $v(x) \ge (1-\epsilon)$. Since $v(\cdot)$ is upper semi-continuous and $X$ is compact, we conclude that there is $x$ such that $v(x) \ge 1$.

Now we get rid of $\rho$. Assume that for any $x$ there exists $0 < r \le R$ such that $v(x,r) < 1$. Consider the function $t(x) = \sup \left\{r ~\middle\vert~ v(x,r) < 1\right\}$. By our assumption, $t(\cdot)$ is defined everywhere and positive. It can be verified that $t(\cdot)$ is lower semi-continuous, so it attains a positive minimum on $X$. Taking $\rho$ less than this minimum, we get a contradiction with the result of the previous paragraph. Therefore, there is $x$ such that $v(x,r) \ge 1$ for all $r \in (0,R]$.
\end{proof}

\appendix
\setcounter{secnumdepth}{0}
\renewcommand{\thesection}{A}
\section{Appendix: definitions of essentiality}
\label{section:comparison}
\setcounter{theorem}{0}

The statements of this section are not used in the proofs of the results listed in the introduction. But they may be useful to put the results in a wider context of systolic inequalities.

%We mention that Gromov \cite{gromov1983} showed that any Riemannian or Finsler metric on an essential complex satisfies the systolic inequality. It can be shown~\cite{babenko1993} that non-essential complexes support metrics with large systole and small volume.

\begin{lemma}
\label{lemma:essential}
The two definitions of $n$-essential for a topological space $X$ and a positive integer $n$ are equivalent:

\begin{enumerate}[label=(\roman*)]
  \item The classifying map $f : X\to K(\pi_1(X), 1)$ cannot be deformed to the $(n-1)$-skeleton of $K(\pi_1(X), 1)$.
  \item $X$ cannot be covered by $n$ or fewer open sets $U_i\subseteq X$ so that for every connected component $C$ of every $U_i$ and every connected component $D$ of $X$ the map $\pi_1(C)\to \pi_1(D)$ is trivial (in this case we call $U_i$ \emph{inessential}).
\end{enumerate}
%i)
%The classifying map $f : X\to K(\pi_1(X), 1)$ cannot be deformed to the $(n-1)$-skeleton of $K(\pi_1(X), 1)$.
%
%ii)
%$X$ cannot be covered by $n$ or fewer open sets $U_i\subseteq X$ so that for every connected component $C$ of every $U_i$ and every connected component $D$ of $X$ the map $\pi_1(C)\to \pi_1(D)$ is trivial (in this case we call $U_i$ \emph{inessential}).
\end{lemma}
\begin{proof}
Set $G := \pi_1(X)$. If (i) fails for a certain deformed map $f$ then one covers the $(n-1)$-skeleton $K^{(n-1)}$ of $K(G, 1)$ with $n$ inessential open sets $K^{(n-1)} = V_1\cup\dots\cup V_n$. This may be done, for example, using the barycentric subdivision of $K^{(n-1)}$ and the proper coloring of its vertices in $n$ colors, then $V_i$ is a union of stars of vertices of color $i$ in the barycentric subdivision. Then the open sets $U_i = f^{-1}(V_i)\subseteq X$ show that (ii) fails as well.

In the opposite direction, we may work with every component of $X$ separately and assume $X$ is connected and has a universal cover. If $X=U_1\cup\dots\cup U_n$ with inessential $U_i$, then the universal covering space $\widetilde X$ equals $\widetilde U_1\cup\dots\cup \widetilde U_n$, so that $G$ acts freely on the components of every $\widetilde U_i$ (this is the meaning of ``inessential'' in terms of the universal cover). Hence there exist $G$-equivariant maps $f_i : \widetilde U_i\to G$ and using a $G$-invariant partition of unity $\rho_i$ subordinate to the open cover $\{\widetilde U_i\}$ we obtain
\[
\widetilde f(x) = \rho_1(x) f_1(x) * \dots * \rho_n(x) f_n(x).
\]
This $\widetilde f$ a $G$-equivariant map from $\widetilde X$ to the join $\underbrace{G*\dots *G}_n$. Passing to quotients, one obtains $f : X\to \underbrace{G*\dots *G}_n/G$, which may be viewed as a map from $X$ to the $(n-1)$-skeleton of $K(G, 1) = \underbrace{G*\dots*G}_\infty /G$, inducing an isomorphims of the fundamental groups. This shows that $X$ is not $n$-essential in terms of maps as well.
\end{proof}

\begin{lemma}
\label{lemma:topological-combinatorial}
If a connected $n$-dimensional complex is $n$-essential as a topological space then it is combinatorially $n$-essential in the sense of Definition \ref{definition:combinatorially-n-essential}.
\end{lemma}

Note that example in Remark \ref{remark:complete-graph} shows that the opposite direction of the lemma is false.

\begin{proof}[Proof of Lemma \ref{lemma:topological-combinatorial}]
A partition of a vertex set $V(X) = Y_1\sqcup\dots\sqcup Y_n$ of a triangulation produces an open cover $X = U(Y_1)\cup\dots\cup U(Y_n)$, where $U(Y_i)$ is the union of open stars of the vertices of $Y_i$, which retracts onto $\langle Y_i\rangle$ and is therefore inessential whenever $\langle Y_i\rangle$ is such. Now we apply the equivalence from Lemma \ref{lemma:essential}.
\end{proof}

\begin{remark}
Lemma \ref{lemma:topological-combinatorial} cannot be extended to more general cell complexes from simplicial complexes. Take a cube $Q_n=[-1,1]^n$. Its vertices are split into two subsets $O$ and $E$ depending on whether the number of $-1$'s is odd or even. No vertex from $O$ is connected by an edge of the cube to another vertex of $O$, and the same for $E$. When $n$ is even, the antipodal involution $\tau(x) = -x$ takes $E$ to $E$ and $O$ to $O$. The quotient $X=\partial Q_n/\tau$ is then a cell complex homeomorphic to $(n-1)$-essential $\mathbb RP^n$. The vertices of $X$ split into two types, $E/\tau$ and $O/\tau$, both inducing no edge and therefore being inessential.
\end{remark}

A weaker version in the opposite direction does hold.

\begin{lemma}
\label{lemma:combinatorial-topological}
If all sufficiently fine (iterated) barycentric subdivisions of a complex $X$ are combinatorially $n$-essential then the complex is $n$-essential as a topological space.\end{lemma}

\begin{proof}
We use Lemma~\ref{lemma:essential} and assume that $X$ is not topologically $n$-essential. Then there is a cover $X=\bigcup_i U_i$ by $n$ or fewer inessential open subsets. It remains to show that some barycentric subdivision of $X$ allows splitting of vertices into $n$ or fewer non-essential vertex sets.

We consider a partition of unity $\sum_i \rho_i \equiv 1$ subordinate to $\{U_i\}$. The compact sets 
\[
F_i = \{x\in X\ |\ \rho_i(x)\ge 1/n\}
\] 
are contained in their respective open sets $U_i$, and $X=\bigcup_i F_i$. Let $\delta>0$ be the smallest over $i$ distance between the disjoint compact sets $F_i$ and $X\setminus U_i$ (assuming some metric on $X$).

Now we take an iterated barycentric subdivision $T$ of $X$ whose all faces have diameter strictly less than $\delta$. 
Partition its vertex set $V(T)$ into a disjoint union $\sqcup V_i=V(T)$ such that $V_i\subset F_i$ for all $i$.
For any $V_i$, the induced subcomplex $\langle V_i\rangle$ is then fully contained in $U_i$ by the choice of $\delta$ and $T$. The restriction of the universal covering map to $\langle V_i\rangle$ is then trivial, since it is trivial over $U_i$. Hence every $V_i$ is an inessential set of vertices.
\end{proof}

%\begin{lemma}
%\label{lemma:cup-essential}
%A cup-essential $n$-complex is $n$-essential.
%\end{lemma}
%
%\begin{proof}
%A cup-essential complex obeys the systolic estimate hence it is essential.
%\end{proof}

\bibliography{../Bib/karasev}
\bibliographystyle{abbrv}

\end{document}